\documentclass[12pt]{article}
\usepackage{amsmath}
\usepackage{amssymb}
\usepackage{latexsym}
\usepackage{amsthm}
\usepackage{mathrsfs}
\usepackage{enumitem}
\usepackage{float}

\usepackage[colorlinks,
            linkcolor=red,
            anchorcolor=blue,
            citecolor=green
            ]{hyperref}
\newtheorem{thm}{Theorem}[section]
\newtheorem{lem}[thm]{Lemma}

\newtheorem{conj}[thm]{Conjecture}

\newcommand{\s}{\mathbf{s}}

\DeclareMathOperator{\des}{des}
\DeclareMathOperator{\Des}{Des}

\newcommand{\sep}{\preceq}

\newcommand{\floor}[1]{\ensuremath{\left\lfloor #1 \right\rfloor}} 

\begin{document}

\begin{center}
{\large \bf  On the Real-rootedness of the Descent Polynomials of $(n-2)$-Stack Sortable Permutations
}
\end{center}

\begin{center}
Philip B. Zhang\\[6pt]

Center for Combinatorics, LPMC-TJKLC\\
Nankai University, Tianjin 300071, P. R. China\\[8pt]

Email: {\tt zhangbiaonk@163.com}
\end{center}

\begin{abstract}
B\'ona conjectured that the descent polynomials on $(n-2)$-stack sortable permutations have only real zeros.
Br{\"a}nd{\'e}n proved this conjecture by establishing a more general result.
In this paper, we give another proof of Br{\"a}nd{\'e}n's result by using the theory of $\s$-Eulerian polynomials recently developed by
Savage and Visontai.

\end{abstract}

\section{Introduction}

Suppose that $w=w_1\cdots w_n$ is a permutation of a set of distinct numbers and
$w_i$ is the maximal number of $\{w_1,\ldots,w_n\}$.
The stack sorting operation $s$ on $w$ can be recursively defined as
$$s(w)=s(w_1\cdots w_{i-1})s(w_{i+1}\cdots w_n)w_i.$$
Let $\mathfrak{S}_n$ denote the set of permutations of $[n]=\{1,2,\ldots,n\}$.
We say that $\sigma\in \mathfrak{S}_n$ is $t$-stack sortable if $s^t(\sigma)$
is the identity permutation. For more information on $t$-stack sortable permutations,
see B\'ona \cite{Bona2002Symmetry}, Knuth \cite{Knuth1969art}, and West \cite{West1990Permutations}.

For $\sigma=(\sigma_1,\sigma_2, \ldots, \sigma_n) \in \mathfrak{S}_n$, let
$$\Des \sigma = \{i \in [n-1] : \sigma_i > \sigma_{i+1}\}$$
denote the set of descents of $\sigma$, and let $\des \sigma = |\Des \sigma|.$
The Eulerian polynomial $A_n(x)$ is usually defined as the descent generating function over $\mathfrak{S}_n$, namely,
\begin{align}\label{eq-eulerionpol}
A_n(x) = \sum_{\sigma\in\mathfrak{S}_n}x^{\des \, \sigma}.
\end{align}

Let $W_{t}(n,k)$ be the number of $t$-stack sortable permutations in $\mathfrak{S}_n$ with $k$ descents, and let
\begin{align*}
W_{n,t}(x)=\sum_{k=0}^{n-1}W_{t}(n,k)x^{k}
\end{align*}
be the descent polynomial over $t$-stack sortable permutations.
B\'ona \cite{Bona2002Symmetry} showed that for fixed $n$ and
$t$ the descent polynomial $W_{n,t}(x)$ is symmetric and unimodal,
and proposed the following conjecture.

\begin{conj}[\cite{Bona2002Symmetry}]
The descent polynomial $W_{n,t}(x)$ has only real zeros for any integer $1 \le t \le n-1$.
\end{conj}

The above conjecture is true for $t=1,2,n-2$, or $n-1$, see Br\"and\'en \cite{Braenden2006linear} and references therein.
In fact, $W_{n,1}(x)$ are the Narayana polynomials and $W_{n,n-1}(x)$ are the Eulerian polynomials, both of which
are known to be real-rooted.
Based on a compact and simple form of $W_{2}(n,k)$ due to Jacquard and Schaeffer \cite{Jacquard1998bijective},
\begin{align*}
W_{2}(n,k)=\frac{(n+k)!(2n-k-1)!}{(k+1)!(n-k)!(2k+1)!(2n-2k-1)!},
\end{align*}
Br\"and\'en proved the real-rootedness of $W_{n,2}(x)$ by using the tool of multiplier sequences.
For $t=n-2$, it is easy to show that
\begin{align*}
W_{n,n-2}(x)=A_{n}(x)-x\, A_{n-2}(x).
\end{align*}
By using certain real-rootedness preserving linear operator, Br\"and\'en  proved the real-rootedness of $W_{n,n-2}(x)$.
Remarkably, Br\"and\'en \cite{Braenden2006linear} obtained the following result.

\begin{thm}[\cite{Braenden2006linear}]\label{branthm}
For any $n \ge 3$ and  $k\ge -2$,
the polynomial
\begin{align}\label{eq-main}
K_n(x)=A_{n}(x)+ k x\,A_{n-2}(x)
\end{align}
has only real zeros.
\end{thm}

The main objective of this paper is to give another proof of the above result by using the theory of $\s$-Eulerian polynomials recently developed by Savage and Visontai \cite{Savage2015s}. The classical Eulerian polynomials $A_n(x)$ are the
$\s$-Eulerian polynomials corresponding to the sequence $\s=(1,2,\ldots)$.
The $\s$-Eulerian polynomials have proven to be a powerful tool for studying the real-rootedness of Eulerian-like polynomials, see also Yang and Zhang \cite{YangMutual}. Instead of directly proving
Theorem \ref{branthm}, we shall prove a slightly general result as shown below.

\begin{thm}\label{main-thm}
For any $n > 3$ and  $k\ge -n$,
the polynomial $A_{n}(x)+ k x\,A_{n-2}(x)$ has only real zeros.
\end{thm}

The remainder of the paper is organized as follows.
In Section 2, we shall review some related concepts and results, and then give a proof of Theorem \ref{main-thm}.
In Section 3, we shall present one open problem.

\section{Proof of Theorem \ref{main-thm}}

The aim of this section is to prove Theorem \ref{main-thm}. The key ingredient of our proof
is the interlacing property of some refined Eulerian polynomials, due to Savage and Visontai \cite{Savage2015s}.

Let us first recall the definition of pairwise interlacing. Given two real-rooted polynomials $f(z)$ and $g(z)$ with positive leading coefficients,  we say that {$g(z)$ interlaces $f(z)$}, denoted $g(z)\sep f(z)$, if
\begin{align*}
\cdots\le s_2\le r_2\le s_1\le r_1,
\end{align*}
where $\{r_i\}$ and $\{s_j\}$ are the sets of zeros
of $f(z)$ and $g(z)$, respectively.
We say that a sequence of real polynomials $(f_{1}(x),\dots,f_{m}(x))$
with positive leading coefficients is pairwise interlacing if $f_{i}(x) \sep f_{j}(x)$ for all $1\le i<j\le m$.
For more information on the theory of pairwise interlacing, see Fisk \cite[Chapter 3]{FiskPolynomials} and Br{\"a}nd{\'e}n \cite[Section 7.8]{Braenden2015Unimodality}.

During their study of the real-rootedness of $\s$-Eulerian polynomials,
Savage and Visontai \cite{Savage2015s} introduced certain refinement of these polynomials
and obtained the pairwise interlacing property of the refined $\s$-Eulerian polynomials.
For the purpose here, we need the following refinement of the Eulerian polynomials $A_n(x)$:
\begin{align}\label{eq-refine-Eulerian}
A_{n,i}(x)=\sum_{{\sigma \in \mathfrak{S}_n} \atop {\sigma_{n}=n-i}} x^{\des\, \sigma}, \quad
\mbox{for } 0\le i \le n-1.
\end{align}
It is clear that $$A_{n}(x) = \sum_{i=0}^{n-1}A_{n,i}(x).$$

We would like to mention that $A_{n,i}(x)$ are just those refined $\s$-Eulerian polynomials considered
by Savage and Visontai, which correspond to the case of $\s=(1,2,\dots)$.
(For this correspondence, see {\cite[Section 3.1]{Savage2015s}}.)
Consequently, we have the following result.

\begin{lem} [{\cite[Lemma 2.1]{Savage2015s}}]
\label{lem:A_pairwise_interlacing}
 For $n\geq2$ and $i=0,1,\ldots,n-1,$
\begin{align}\label{eq:E_n,i}
A_{n,i}(x)=x\sum_{j=0}^{i-1}A_{n-1,j}(x)+\sum_{j=i}^{n-2}A_{n-1,j}(x)
\end{align}
with the initial condition $A_{1,0}(x)=1$.
Furthermore, the polynomial sequence
$\{A_{n,i}(x)\}_{i=0}^{n-1}$ is pairwise interlacing.
\end{lem}

As remarked in \cite{Savage2015s}, the polynomials $A_{n,i}(x)$ and their recurrence relation  also
appeared in \cite{Nevo2011} in a slightly different form. With the aid of this recurrence, we can
give an expression of $A_{n}(x)$ in terms of $A_{n-1,j}(x)$, which plays an important role in our proof of
Theorem \ref{main-thm}.

\begin{lem} For any integer $n\geq 2$, we have
\begin{align}\label{eq:S}
A_{n}(x)= & \sum_{j=0}^{n-2}\big( (n-j-1)x+j+1\big)A_{n-1,j}(x).
\end{align}
\end{lem}

\begin{proof} The recurrence relation (\ref{eq:E_n,i}) can be expressed in a matrix form as follows:
\begin{align}\label{A-recurrence-matrix}
\left(\begin{array}{c}
  A_{n,0}(x) \\
  A_{n,1}(x) \\
  A_{n,2}(x) \\
  \vdots \\
  A_{n,n-1}(x) \\
\end{array}\right)
 =
\left(\begin{array}{cccc}
  1 & 1 & \cdots & 1 \\
  x & 1 & \cdots & 1 \\
  x & x & \cdots & 1 \\
  \vdots & \vdots &   & \vdots \\
  x & x & \cdots & x \\
\end{array}\right)
 \cdot
\left(\begin{array}{c}
  A_{n-1,0}(x) \\
  A_{n-1,1}(x) \\
  A_{n-1,2}(x) \\
  \vdots \\
  A_{n-1,n-2}(x) \\
\end{array}\right).
\end{align}

It is readily to see that 
\begin{align*}
A_{n}(x)=&\sum_{i=0}^{n-1}A_{n,i}(x)= \left(1,1,\ldots,1\right)
 \cdot
\left(\begin{array}{c}
  A_{n,0}(x) \\
  A_{n,1}(x) \\
  A_{n,2}(x) \\
  \vdots \\
  A_{n,n-1}(x) \\
\end{array}\right)\\[6pt]
= & \left(1,1,\ldots,1\right) \cdot \left(\begin{array}{cccc}
  1 & 1 & \cdots & 1 \\
  x & 1 & \cdots & 1 \\
  x & x & \cdots & 1 \\
  \vdots & \vdots &   & \vdots \\
  x & x & \cdots & x \\
\end{array}\right)
 \cdot
\left(\begin{array}{c}
  A_{n-1,0}(x) \\
  A_{n-1,1}(x) \\
  A_{n-1,2}(x) \\
  \vdots \\
  A_{n-1,n-2}(x) \\
\end{array}\right)\\[6pt]
= & \big((n-1)x+1,(n-2)x+2,\ldots,x+(n-1)\big)
 \cdot
\left(\begin{array}{c}
  A_{n-1,0}(x) \\
  A_{n-1,1}(x) \\
  A_{n-1,2}(x) \\
  \vdots \\
  A_{n-1,n-2}(x) \\
\end{array}\right)\\[6pt]
= & \sum_{j=0}^{n-2}\big( (n-j-1)x+j+1\big)A_{n-1,j}(x),
\end{align*}
as desired.  This completes the proof.
\end{proof}

To prove Theorem \ref{main-thm}, we also need the following result due to Haglund, Ono, and Wagner \cite{Haglund1999Theorems}.

\begin{thm}[{\cite[Lemma~8]{Haglund1999Theorems}}]\label{Haglund1999}
Let $f_1(x), \dots, f_m(x)$ be real-rooted polynomials with nonnegative coefficients, and let  $a_1, \dots, a_m \ge 0$  and $b_1, \dots, b_m \ge 0$ be such that
$a_{i}b_{i+1} \ge b_{i}a_{i+1}$ for all $1 \le i \le m-1.$ If the sequence $(f_1(x), \dots, f_m(x))$ is pairwise interlacing, then
$$\sum_{i=1}^{m}a_{i}f_{i}(x) \sep \sum_{i=1}^{m}b_{i}f_{i}(x).$$
\end{thm}

We would like to point out that the above result can be taken as a special case of \cite[Corollary 7.8.6]{Braenden2015Unimodality}, and the corresponding matrix is 
$$
G=\begin{pmatrix}
a_1 & a_2 & \cdots & a_m\\
b_1 & b_2 & \cdots & b_m
\end{pmatrix}.
$$
In fact, the conditions $a_i\geq 0, b_i\geq 0, a_{i}b_{i+1} \ge b_{i}a_{i+1}$ imply that all minors of $G$ are nonnegative.

Now we are in the position to prove the main result of this paper.

\noindent\textit{Proof of Theorem \ref{main-thm}.}
By \eqref{eq:S}, we have
\begin{align*}
K_n(x)&=A_{n}(x)+ k  x\,A_{n-2}(x)\\[5pt]
& =\sum_{j=0}^{n-2}\big((n-j-1)x+j+1\big)A_{n-1,j}(x)+ k  x\,  A_{n-2}(x)\\[5pt]
& =\big((n-1)x+1\big)A_{n-1,0}(x) +\sum_{j=1}^{n-3}\big((n-j-1)x+j+1\big)A_{n-1,j}(x)\\[5pt]
& \quad + \big(x+n-1\big)A_{n-1,n-2}(x)+ k  x\,  A_{n-2}(x).
\end{align*}

Note that
\begin{align*}
 A_{n-1,0}(x) = A_{n-2}(x) \quad \mbox{and} \quad A_{n-1,n-2}(x) = xA_{n-2}(x).
\end{align*}
By the definition of the refined Eulerian polynomials given by \eqref{eq-refine-Eulerian}, these two identities can be
easily derived from the following facts: given a permutation $\sigma=(\sigma_1,\sigma_2, \ldots, \sigma_{n-1}) \in \mathfrak{S}_{n-1}$,
if $\sigma_{n-1} = n-1$, then $n-2$ can not be a descent of $\sigma$; while if $\sigma_{n-1} = 1$, then $n-2$ must be a descent
of $\sigma$. With these identities, we get that
\begin{align*}
K_n(x)
=&  \big((n+\frac{k}{2}-1)x+1\big)A_{n-1,0}(x) \\
& + \sum_{j=1}^{n-3}\big((n-j-1)x+j+1\big)A_{n-1,j}(x)\\
& + \big(x+n+\frac{k}{2}-1\big)A_{n-1,n-2}(x).
\end{align*}

Now we shall use Theorem \ref{Haglund1999} to obtain the real-rootedness of $K_n(x)$. To this end,
let $m=n-1$, $f_i(x)=A_{n-1,i-1}(x)$ for $1\le i\le m$ and
$$
a_i=\left\{
\begin{array}{ll}
n+\frac{k}{2}-1, & i=1,\\
n-i, & 2\leq i\leq m-1,\\
1, & i=m,
\end{array}
\right.
$$
and
$$
b_i=\left\{
\begin{array}{ll}
1, & i=1,\\
i, & 2\leq i\leq m-1,\\
n+\frac{k}{2}-1, & i=m.
\end{array}
\right.
$$
By Lemma \ref{lem:A_pairwise_interlacing}, we see that
the polynomial sequence $\{f_i(x)\}_{i=1}^{m}$ is pairwise interlacing.
Moreover, the numbers $a_i$ and $b_i$ satisfy the conditions of Theorem \ref{Haglund1999}, since, by our hypothesis,
\begin{align*}
 & a_{1}b_{2}-b_{1}a_{2}  = (n+\frac{k}{2}-1)2-(n-2)=n+k\geq 0, \\
 & a_{m-1}b_{m}-b_{m-1}a_{m}  = 2(n+\frac{k}{2}-1)-(n-2)=n+k\geq 0,
\end{align*}
and, for $2\le i \le m-2$,
\begin{align*}
 a_{i}b_{i+1}-b_{i}a_{i+1}  = (n-i)(i+1)-i(n-i-1)=n>0.
\end{align*}
Therefore, we have
$$\sum_{i=1}^{m}a_{i}f_{i}(x) \sep \sum_{i=1}^{m}b_{i}f_{i}(x).$$
Since all the zeros of these two polynomials are real and nonpositive, we get that
$$\sum_{i=1}^{m}b_{i}f_{i}(x)\sep x \sum_{i=1}^{m}a_{i}f_{i}(x).$$
Thus, the polynomial $$K_n(x) = \sum_{i=1}^{m}\left(a_{i}x+ b_{i}\right)f_{i}(x)$$
has only real zeros, as the sum of interlacing polynomials must be real-rooted.
This completes the proof.
\qed

\section{One open problem}

We have shown that, for any $n > 3$ and  $k\ge -n$,
the polynomial $K_{n}(x)$ in \eqref{eq-main} has only real zeros.
Stanley \cite{StanleyPrivate} advised us to further study under what conditions the polynomial $K_n(x)$ has only real zeros. Inspired by his suggestion, we first considered when the polynomial $K_n(x)$ has all its zeros both distinct and real.

Let us first recall a useful criterion for determining whether a polynomial of degree $n$ has $n$ distinct real zeros.
Suppose that
$$f(x)=\sum_{i=0}^{n}a_{n-i}x^{i}$$
and
$$g(x)=\sum_{i=0}^{n}b_{n-i}x^{i}$$
are two polynomials with $a_0\ne 0$.
For any $1\leq k\leq n$, let
\begin{align*}
  \Delta_{2 k}\left( f(x),g(x) \right)= \det \left(
  \begin{array}{ccccc}
   a_0 & a_1 & a_2 & \dots & a_{2k-1}\\
   b_0 & b_1 & b_2 & \dots & b_{2k-1}\\
    0  & a_0 & a_1 & \dots & a_{2k-2}\\
    0  & b_0 & b_1 & \dots & b_{2k-2}\\
  \vdots&\vdots&\vdots&  & \vdots\\
    0  &  0  &  0  & \dots & b_{k}\\
  \end{array}\right)_{2k \times 2k}.
\end{align*}
These determinants are known as the Hurwitz determinants of $f(x)$ and $g(x)$.
The following result, essentially due to Borchardt and  Hermite \cite[pp. 349]{Rahman2002Analytic}, shows that we can determine whether
$f(x)$ has only real and distinct zeros by inspecting the signs of $\Delta_{2k}(f(x),f'(x))$.

\begin{thm}\label{Hurwitz criterion}
Suppose that $f(x)$ is a real polynomial of degree $n$ with $a_0>0$. Then
$f(x)$ has $n$ distinct real zeros  if and only if the corresponding Hurwitz determinants satisfy
\begin{align}\label{Hurwitz_det}
\Delta_{2k}(f(x),f'(x))>0, \quad \text{for every } 1\leq k \leq n.
\end{align}
\end{thm}

When $n$ is not too large, we can use the above characterization to find the range of admissible values of $k$ for which the polynomial $K_n(x)$ has all its zeros both real and distinct.
For small $n$, we found that there exist two real numbers $\omega_n$ and $\Omega_n$ such that the polynomial $K_n(x)$ has all its zeros both real and distinct if and only if  $k \in (-\infty,\omega_n) \cup (\Omega_n,+\infty)$.  In the following table, we list the values of $\omega_n$ and $\Omega_n$ for $3\leq n\leq 18$.

\begin{center}
\begin{tabular}{|r|c|c|}
\hline  $n$ & $\omega_n$ & $\Omega_n$\\  \hline
  $3$ & $-6$  & $-2$ \\
  $4$ & $-12$  & $-8$ \\
  $5$ & $-20$  & $-8$ \\
  $6$ & $-30$  & $-17$ \\
  $7$ & $-42$  & $-17$ \\
  $8$ & $-56$  & $-496/17$ \\
  $9$ & $-72$  & $-496/17$ \\
 $10$ & $-90$  & $-1382/31$ \\
 $11$ & $-110$  & $-1382/31$ \\
 $12$ & $-132$  & $-43688/691$ \\
 $13$ & $-156$  & $-43688/691$ \\
 $14$ & $-182$  & $-929569/10922$ \\
 $15$ & $-210$  & $-929569/10922$ \\
 $16$ & $-240$  & $-102473312/929569$ \\
 $17$ & $-272$  & $-102473312/929569$ \\
 $18$ & $-306$  & $-443861162/3202291$\\  \hline
\end{tabular}
\end{center}

For $3\leq n\leq 18$, we observed that $\omega_n=-n(n-1)$. Although $\Omega_n$ are rational numbers,
the consecutive product $\prod_{i=1}^n \Omega_{2i+1}$ turns out to be an integer number, which coincides with
the $(n+1)$-th tangent number up to a sign, see \cite[A000182]{SloaneLine}.
It is known that the $n$-th tangent number $T_n$ counts the number of up-down (or down-up) permutations of $[2n-1]$.
More computer evidence suggests the following conjecture.

\begin{conj}
For any $n\geq 3$, the polynomial $K_n(x)$ has only real and distinct zeros if and only if
$k \in (-\infty,-n(n-1)) \cup (-a(\floor{n/2}),+\infty)$, where $a(n) = T_{n+1}/T_{n}$.
\end{conj}

\vskip 3mm
\noindent {\bf Acknowledgments.} We wish to thank Professor Richard Stanley for his helpful comments.
We also would like to thank the referee for recommending various improvements of the manuscript.
This work was supported by the 973 Project, the PCSIRT Project of the Ministry of Education and the National Science Foundation of China.


\end{document}